\documentclass[twoside,8pt,B5]{article}

\usepackage{graphicx,amsmath,latexsym,amssymb,amsthm,fancyhdr}
\usepackage{enumerate}

\def\C{{\mathbb{C}}}

\def\N{{\mathbb{N}}}
\def\P{{\mathbb{P}}}

\def\ppmod{\mkern-16mu \pmod}

\font\chuto=cmbx10 at 16pt  \font\chudams=cmbxsl8

\newtheorem{thm}{Theorem}[section]

\theoremstyle{definition}
\newtheorem{defn}[thm]{Definition}
\newtheorem{rem}[thm]{Remark}

\numberwithin{equation}{section}

\fancypagestyle{plain}{%
\fancyhead[L]{\setlength{\unitlength}{1cm}
\begin{picture}(0,0)
\put(-0.2,1.0){\chudams International Conference in Number Theory and Applications 2012}
\put(-0.2,0.5){\chudams Department of Mathematics, Faculty of Science, Kasetsart University}
\put(-0.2,0.0){\chudams Speaker: P.~Haukkanen}  
\put(-.2,-0.2){\line(1,0){11.3}}
\put(-.2,-0.27){\line(1,0){11.3}}
\put(-.2,-0.3){\line(1,0){11.3}}
\put(-.2,-0.32){\line(1,0){11.3}}
\end{picture}\hfill \thepage}
\fancyfoot[C]{} 

}
\def\ppmod{\mkern-16mu \pmod}
\newcommand*{\Z}{\mathbb{Z}}

\begin{document}

\setcounter{page}{1}

\thispagestyle{plain}
\centerline {~}
\centerline {\bf \chuto Extensions of the class of multiplicative functions}
\medskip

\fancyhead[RE]{\small{\thepage \hfill \textit{P.~Haukkanen}}}  
\fancyhead[LO]{\small{\textit{Extensions of the class of multiplicative functions} \hfill \thepage}}

\renewcommand{\thefootnote}{\fnsymbol{footnote}}

\vskip.8cm
\centerline {Pentti Haukkanen{\footnote{\textit{Corresponding author}}}{\footnote{\textit{The author is supported by The Magnus Ehrnrooth Foundation}}}
}

\renewcommand{\thefootnote}{\arabic{footnote}}

\vskip.5cm
\centerline{School of Information Sciences, 
FI-33014 University of Tampere, Finland
}
\centerline{\texttt{pentti.haukkanen@uta.fi}
}

\vskip .5cm

\begin{abstract}
We consider the classes of quasimultiplicative, semimultiplicative and Selberg multiplicative functions as extensions of the class of multiplicative functions. 
We apply these concepts to Ramanujan's sum and its analogue with respect to regular integers (mod $r$). 
\end{abstract}

\medskip

\noindent
{\bf Mathematics Subject Classification:} 11A25, 11L03 

\smallskip
\noindent
{\bf Keywords:} quasimultiplicative function, semimultiplicative function,  Selberg multiplicative function, Ramanujan's sum, regular integer



\section{Introduction}

An arithmetical function $f\colon\N\to\C$  is said to be
multiplicative if  $f(mn)=f(m)f(n)$
for all $m, n\in\N$ with $(m, n)=1$. 
These functions play a central role in number theory. 
The works of E. T. Bell and R. Vaidyanathaswamy are 
prominent in the history of multiplicative functions, 
see e.g. \cite{B,V}. 

Many of the classical arithmetical functions are multiplicative, e.g. the M\"{o}bius function, Euler's totient function and the divisor functions. 
On the other hand, multiplicative functions have some weak points, 
e.g.,  they are destroyed by compositions such as $cf(n), f(kn), f(k/n), f(n/k), f([k, n])$, where $[k, n]$ is the lcm of $k$ and $n$. 
This has led to certain extensions of the class of multiplicative functions. 
In this paper we introduce quasimultiplicative, semimultiplicative and Selberg multiplicative functions, see \cite{L,Rearick66a,Selberg}. 
As a motivation of these concepts we also consider multiplicative properties of Ramanujan's sum and its analogue with respect to regular integers \cite{HT}. 

There are also important subclasses of the class of multiplicative functions in the number theoretic literature, e.g., the class of rational arithmetical functions, see \cite{LP}. We do not consider these classes in this paper. 

\newpage




\section{Extensions of multiplicative functions} 

\subsection{Extensions of multiplicative functions of one variable}

The usual definition of a multiplicative function is as follows: 

\begin{defn}
An arithmetical function $f\colon\N\to\C$  is 
{\sl multiplicative} if 
\begin{equation}
f(mn)=f(m)f(n)
\end{equation}
for all $m, n\in\N$ with $(m, n)=1$. 
\end{defn}

It is easy to see that a multiplicative function $f$ is totally determined by its values at prime powers. To be more precise, 
an arithmetical function $f$ is multiplicative if  and only if  
\begin{equation}
f(n)=\prod_{p\in\P}
\left(f(p^{\nu_p(n)}\right),  
\end{equation}
where $\nu_p(n)$ is the exponent of $p$ in the canonical factorization of $n$. 
If $f$ is a multiplicative function not identically zero, then  $f(1)=1$.
 
The usual multiplicativity can be easily destroyed, for instance, by multiplying the function values with a constant $(\ne 0, 1)$. 
This leads to the concept of a quasimultiplicative function. 

\begin{defn}
An arithmetical function $f\colon\N\to\C$  is 
{\sl quasimultiplicative} if there exists a nonzero constant $c$ such that  
\begin{equation}
c\,f(mn)=f(m)f(n)
\end{equation}
for all $m, n\in\N$ with $(m, n)=1$. 
\end{defn}

It is easy to see that an arithmetical function $f$ not identically zero  is quasimultiplicative if  and only if  $f(1)\ne 0$ and 
\begin{equation}
f(1)f(mn)=f(m)f(n)
\end{equation}
for all $m, n\in\N$ with $(m, n)=1$. Then $c=f(1)$. 
Quasimultiplicative functions are multiplicative functions 
multiplied by a  constant.  
An arithmetical function $f$ not identically zero is quasimultiplicative if  and only if $f(1)\ne 0$ and $f/f(1)$ is multiplicative. 

D. B. Lahiri \cite{L} introduced the concept of quasimultiplicative functions as a special case of hypomultiplicative functions. 
Lahiri noted that  $r_2(n)$, $r_4(n)$ and $r_8(n)$ are examples of  quasimultiplicative functions, where  $r_s(n)$ is the number of representations of $n$ as the sum of $s$ squares.

The concept of a multiplicative function or a quasimultiplicative function is not satisfactory in the sense that compositions such as $f(kn), f(k/n), f(n/k), f([k, n])$ $(k\in\N\setminus\{1\})$ preserve neither multiplicativity nor quasimultiplicativity. This has led to the concepts of semimultiplicative and Selberg multiplicative functions. 

We first introduce the concept of a semimultiplicative function. This concept is due to David Rearick \cite{Rearick66a} and is also considered, e.g., in the book by R. Sivaramakrishnan \cite{Si}. For further material, see 
\cite{Hau,HT,Rearick66b}. 

\begin{defn} \label{de:semimu}
An arithmetical function $f\colon\N\to\C$ is said to be
{\sl semimultiplicative} 
if there exists a nonzero constant $c$, a positive integer $a$ and
a multiplicative function $f_m$ such that
\begin{equation}\label{e:sem-rea}
f(n)=c f_m(n/a)
\end{equation}
for all $n\in\N$. 
(Here $f_m(x)=0$ if $x$ is not a positive integer.)
\end{defn}

Note that if $f$ is not identically zero, then  $f(a)=c\ne 0$ and $f(n)=0$ for $n<a$. The following theorem follows easily from the definition.

\begin{thm}\label{th:semi1}
An arithmetical function $f$ not identically zero is semimultiplicative if  and only if  
there exists a positive integer $a$ such that 
$f(a)\ne 0$ and 
$f(a x)$ is an arithmetical function (i.e. $f(a x)=0$ if $x$ is not a positive integer)  and 
$$
{f(a n)\over f(a)} 
$$
is multiplicative in $n$. 
\end{thm}

This can also be written in the following form. 

\begin{thm}\label{th:semi2}
An arithmetical function $f$ not identically zero is semimultiplicative if  and only if  
there exists a positive integer $a$ such that 
$f(a)\ne 0$ and 
$f(n)=0$ whenever $a\nmid n$ 
and
$$
f(a) f(a m n) = f(a m)f(a n) \text{  whenever  }  (m, n)=1. 
$$
\end{thm}

A further characterization is as follows.  This nice identity was proved by Rearick \cite{Rearick66a}. 

\begin{thm}
An arithmetical function $f$ is semimultiplicative  if and only if
\begin{equation}
f(m)f(n)=f((m, n))f([m, n])
\end{equation}
for all $m, n\in\N$.  
\end{thm}

Semimultiplicative  functions $f$ with $a=1$ and $c=1$ (i.e. $f(1)=1$) are multiplicative functions and 
semimultiplicative  functions $f$ with $a=1$ (i.e. $f(1)\ne 0$) are quasimultiplicative functions.

We now go to the concept of Selberg multiplicative functions. 
The term ``Selberg multiplicative"  was given in \cite{BHS,HHS} 
in the honor of Selberg who introduced these functions in 
\cite{Selberg}. 
Selberg \cite{Selberg} said on the concept of the usual multiplicative functions that  ``{\sl I have never been very satisfied with this definition and would prefer to define a multiplicative function as follows}".   

\begin{defn}\label{de:Sel}
An arithmetical function $f\colon\N\to\C$ is  {\sl Selberg multiplicative} if 
for each prime $p$ there exists $F_p\colon\N_0\to\C$ 
with $F_p(0)=1$ for
all but finitely many primes $p$ such that
\begin{equation}\label{e:Sel}
f(n)=\prod_{p\in\P} F_p(\nu_p(n))
\end{equation}
for all $n\in\N$, where $\nu_p(n)$ is the exponent of $p$ in the canonical factorization of $n$. 
\end{defn}

Multiplicative functions $f$ are Selberg multiplicative with 
\begin{equation}
F_p(\nu_p(n))=f(p^{\nu_p(n)}).   
\end{equation}
Quasimultiplicative functions $f$ are Selberg multiplicative with 
Selberg factorization 
\begin{equation}
f(n)=f(1)\prod_{p\in\P}
\left(\frac{f(p^{\nu_p(n)})}{f(1)}\right)
\end{equation}
provided that $f(1)\ne 0$. 

The following theorem was given in \cite{Hau}. 

\begin{thm}
An arithmetical function is Selberg multiplicative if and only
if it is semimultiplicative. In fact, a semimultiplicative function $f$ possesses a Selberg factorization as
\begin{equation}
f(n)=f(a)\prod_{p\in\P}
\left(\frac{f(ap^{\nu_p(n)-\nu_p(a)})}{f(a)}\right).
\end{equation}
\end{thm}

Selberg multiplicative (or semimultiplicative) functions possess the following useful properties. 

\begin{description}

\item{(i)} Dirichlet convolution preserves  Selberg multiplicativity. 

\item{(ii)} Usual product preserves  Selberg multiplicativity. 

\item{(iii)} Compositions 
$$
f(kn), f(k/n), f(n/k), f((k, n)), f([k, n]),\quad k\in\N, 
$$ 
preserve Selberg multiplicativity. 

\end{description}

Property (i) is proved in \cite{Rearick66a}, property (ii) is easy to prove and properties (iii) are presented in \cite{Rearick66a,Selberg} without proofs. 

We close this section with a short summary: 

\begin{enumerate}[(a)]
\item Multiplicative functions are quasimultiplicative, and 
quasimultiplicative functions are semimultiplicative. 
\item Semimultiplicative functions $f$ with $f(1)\ne 0$ are quasimultiplicative. 
\item Quasimultiplicative functions $f$ with $f(1)=1$ are multiplicative. 
\item Selberg multiplicative functions are the same as semimultiplicative 
functions. 
\end{enumerate}

\subsection{Extensions of multiplicative functions of several variables}

The usual notion of multiplicative functions of several variables is presented in the following definition. 

\begin{defn}
An arithmetical function $f\colon\N^u\to\C$ of $u$ 
variables is {\sl  multiplicative} in $n_1, n_2,\ldots, n_u$ 
if 
\begin{equation}
f(n_1m_1, n_2m_2,\ldots, n_um_u)
=f(n_1, n_2,\ldots, n_u)f(m_1, m_2,\ldots, m_u)
\end{equation}
for all $n_1, n_2,\ldots, n_u\in\N$ and $m_1, m_2,\ldots, m_u\in\N$ 
with $(n_1 n_2\cdots n_u, m_1m_2\cdots m_u)=1$. 
\end{defn}  

This definition means that a multiplicative function $f$ is completely determined by its values 
$f(p^{a_1}, p^{a_2},\ldots, p^{a_u})$ at prime powers. 
In fact, 
an arithmetical $f$  is multiplicative in $n_1, n_2,\ldots, n_u$ if  and only if  
\begin{equation}
f(n_1, n_2,\ldots, n_u)
=\prod_{p\mid n_1n_2\cdots n_u }  f(p^{\nu_p(n_1)}, 
p^{\nu_p(n_2)},\ldots, p^{\nu_p(n_u)})
\end{equation}
for all $n_1, n_2,\ldots, n_u\in\N$. 
If a multiplicative function $f$ is not identically zero, then $f(1, 1,\ldots, 1)=1$. 

The concept of multiplicative functions is easy to generalize to the concept of quasimultiplicative functions.  

\begin{defn} 
An arithmetical function $f\colon\N^u\to\C$ of $u$ 
variables is  {\sl quasimultiplicative} 
in $n_1, n_2,\ldots, n_u$ if 
there exists a nonzero constant $c$ such that 
\begin{equation}
c\, f(n_1m_1, n_2m_2,\ldots, n_um_u) 
=
f(n_1, n_2,\ldots, n_u)f(m_1, m_2,\ldots, m_u)
\end{equation}
for all $n_1, n_2,\ldots, n_u\in\N$ and $m_1, m_2,\ldots, m_u\in\N$ 
with $(n_1 n_2\cdots n_u, m_1m_2\cdots m_u)=1$. 
\end{defn} 

It is easy to see that an arithmetical function $f$ not identically zero is quasimultiplicative if and only if $f(1, 1,\ldots, 1)\ne0$ and 
\begin{equation}
f(1, 1,\ldots, 1) f(n_1m_1, n_2m_2,\ldots, n_um_u) 
=
f(n_1, n_2,\ldots, n_u)f(m_1, m_2,\ldots, m_u)
\end{equation}
for all $n_1, n_2,\ldots, n_u\in\N$ and $m_1, m_2,\ldots, m_u\in\N$ 
with $(n_1 n_2\cdots n_u, m_1m_2\cdots m_u)=1$.  

Quasimultiplicative $f$ with $f(1, 1,\ldots, 1)=1$ are multiplicative functions, and an arithmetical function $f$ not identically zero  is quasimultiplicative if  and only if  $f(1, 1,\ldots, 1)\ne 0$ and $f/f(1, 1,\ldots, 1)$ is multiplicative. 

One of Selberg's motivations to define multiplicative functions via Definition \ref{de:Sel} was that 
this concept has a natural generalization to arithmetical functions of several variables. 

\begin{defn} 
An arithmetical function $f\colon\N^u\to\C$ of $u$ 
variables is  {\sl Selberg multiplicative}
in $n_1, n_2,\ldots, n_u$ if 
for each prime $p$ there exists $F_p\colon\N_0^u\to\C$ 
with $F_p(0, 0,\ldots, 0)=1$ for
all but finitely many primes $p$ such that
\begin{equation}\label{e:Sel-exp}
f(n_1, n_2,\ldots, n_u)=\prod_{p\in\P} F_p(\nu_p(n_1), 
\nu_p(n_2),\ldots, \nu_p(n_u))
\end{equation}
for all $n_1, n_2,\ldots, n_u\in\N$.
\end{defn}

Multiplicative functions $f$ are Selberg multiplicative with 
\begin{equation}
F_p(\nu_p(n_1), \nu_p(n_2),\ldots, \nu_p(n_u))=f(p^{\nu_p(n_1)}, 
p^{\nu_p(n_2)},\ldots, p^{\nu_p(n_u)}). 
\end{equation} 
Quasimultiplicative functions are Selberg multiplicative with
Selberg factorization 
\begin{equation}\label{eq:Sel-factor}
f(n_1, n_2,\ldots, n_u)= f(1, 1,\ldots, 1)\prod_{p\in\P}
\left(\frac{f(p^{\nu_p(n_1)}, p^{\nu_p(n_2)},\ldots, p^{\nu_p(n_u)})}
{f(1, 1,\ldots, 1)}\right)
\end{equation}
provided that $f(1, 1,\ldots, 1)\ne 0$. 
Semimultiplicative functions of several variables have not hitherto been considered in the literature. We suggest the following definition which reduces to the concept of semimultiplicative functions of 
Rearick for $u=1$.  

\begin{defn} 
An arithmetical function $f\colon\N^u\to\C$ of $u$ 
variables is {\sl semimultiplicative}
in $n_1, n_2,\ldots, n_u$ 
if there exists a nonzero constant $c$,  positive integers $a_1, a_2, \ldots, a_u$ and a multiplicative function $f_m$ such that
\begin{equation}
f(n_1, n_2,\ldots, n_u)=c\  f_m(n_1/a_1, n_2/a_2,\ldots, n_u/a_u)
\end{equation}
for all $n_1, n_2,\ldots, n_u\in\N$. 
\end{defn}

Note that if $f$ is not identically zero, then $f(a_1, a_2,\ldots, a_u)=c\ne 0$, and 
$f(n_1, n_2,\ldots, n_u)=0$ if $n_i<a_i$ for some $i=1, 2, \ldots, u$. 
A generalization of Theorem \ref{th:semi1} can be written as follows:  

\begin{thm}
An arithmetical function $f$ not identically zero is semimultiplicative 
if  and only if   there exist positive integers $a_1, a_2, \ldots, a_u$ such that 
$f(a_1, a_2,\ldots, a_u)\ne 0$ and
$f(a_1 x_1, a_2 x_2,\ldots, a_u x_u)$ is an arithmetical function (i.e. 
$f(a_1 x_1, a_2 x_2,\ldots, a_u x_u)=0$ if $x_1, x_2,\ldots, x_u$ are not positive integers) and 
$$
{f(a_1 n_1, a_2 n_2,\ldots, a_u n_u)\over f(a_1, a_2,\ldots, a_u)} 
$$
is multiplicative in $n_1, n_2,\ldots, n_u$. 
\end{thm}

\begin{proof} 
Take 
$$
f_m(n_1, n_2,\ldots, n_u)
={f(a_1 n_1, a_2 n_2,\ldots, a_u n_u)\over f(a_1, a_2,\ldots, a_u)}. 
$$
\end{proof}

Theorem \ref{th:semi2} can be generalized as follows:  

\begin{thm}
An arithmetical function $f$ not identically zero is semimultiplicative 
if  and only if  there exist positive integers $a_1, a_2, \ldots, a_u$ such that  
$f(a_1, a_2,\cdots, a_u)\ne 0$ and
$f(n_1, n_2,\ldots, n_u)=0$ whenever 
$a_i\nmid n_i$ for some $i=1, 2,\ldots, u$ 
and 
\begin{equation}\label{eq:mult-semi}
f(a_1, a_2,\ldots, a_u) f(a_1 m_1 n_1, a_2 m_2 n_2,\ldots, a_u m_u n_u)
=   f(a_1 m_1, a_2 m_2,\ldots, a_u m_u)
 f(a_1 n_1, a_2 n_2,\ldots, a_u n_u)
\end{equation}
for all $n_1, n_2,\ldots, n_u\in\N$ and $m_1, m_2,\ldots, m_u\in\N$ 
with $(n_1 n_2\cdots n_u, m_1m_2\cdots m_u)=1$. 
\end{thm}

\begin{thm}
Each semimultiplicative function $f$ is Selberg multiplicative and possesses a Selberg factorization as
\begin{equation}\label{eq:Sel-factor-u}
f(n_1, n_2,\ldots, n_u)= f(a_1, a_2,\ldots, a_u)\prod_{p\in\P}
\left(\frac{f(a_1p^{\nu_p(n_1)-\nu_p(a_1)},a_2p^{\nu_p(n_2)-\nu_p(a_2)},\ldots, a_up^{\nu_p(n_u)-\nu_p(a_u)})}
{f(a_1, a_2,\ldots, a_u)}\right)
\end{equation}
provided that $f(a_1, a_2,\ldots, a_u)\ne 0$. 
\end{thm}

\begin{proof} 
Assume that $a_i\nmid n_i$ for some $i=1, 2,\ldots, u$. 
Then $f(n_1, n_2,\ldots, n_u)=0$. 
Further, $\nu_p(n_i)-\nu_p(a_i)<0$ and therefore 
$a_i\nmid a_i p^{\nu_p(n_i)-\nu_p(a_i)}$. 
This shows that 
$$
f(a_1p^{\nu_p(n_1)-\nu_p(a_1)},a_2p^{\nu_p(n_2)-\nu_p(a_2)},
\ldots, a_up^{\nu_p(n_u)-\nu_p(a_u)})=0,  
$$
and thus the right-hand side of $(\ref{eq:Sel-factor-u})$ is also zero. 
So, $(\ref{eq:Sel-factor-u})$ holds in this case. 

Assume that $a_i\mid n_i$ for all $i=1, 2,\ldots, u$. 
Then  $n_i=a_i m_i$ for all $i=1, 2,\ldots, u$. 
Thus, applying $(\ref{eq:mult-semi})$ we obtain 
\begin{eqnarray*}
f(n_1, n_2,\ldots, n_u)
&=&f(a_1 m_1, a_2 m_2,\ldots, a_u m_u)\\
&=&f(a_1, a_2,\ldots, a_u)\prod_{p\in\P}
\left(\frac{f(a_1p^{\nu_p(m_1)},a_2p^{\nu_p(m_2)},\ldots, a_up^{\nu_p(m_u)})}
{f(a_1, a_2,\ldots, a_u)}\right). 
\end{eqnarray*}
This completes the proof. 
\end{proof}

\begin{rem}
There exist Selberg multiplicative functions of several variables 
that are not semimultiplicative. 
For example, let  $f\colon\N^2\to\C$ be a Selberg multiplicative function such that 
$F_p(s_1, s_2)=1$ except for that $F_2(0, 0)=0$. 
This means that 
\begin{eqnarray*}
f(n_1, n_2)=
\begin{cases}
0 \text{  if $2\nmid n_1$ and $2\nmid n_2$,} \\
1 \text{  otherwise}. 
\end{cases}
\end{eqnarray*}
Thus, for instance, $f(1, 1)=0, f(1, 2)=f(2, 1)=1$. 
Then $f$ is not semimultiplicative. 
In fact, 
since $f(1, 2)\ne 0$, we have $a_1=1$ and $a_2\in\{1, 2\}$, 
and 
since $f(2, 1)\ne 0$, we have $a_1\in\{1, 2\}$ and $a_2=1$. 
Therefore $a_1=1$ and $a_2=1$, but this is impossible, since $f(1, 1)=0$. 
\end{rem}

\begin{rem}
It is easy to see that Selberg multiplicative functions of several variables preserve 
the Dirichlet convolution, the usual product and also compositions like in item (iii) in Section 2.1. It is likewise easy to see that semimultiplicative functions of several variables preserve the Dirichlet convolution but it is not known whether semimultiplicative functions of several variables preserve the usual product and compositions like in item (iii) in Section 2.1. 
\end{rem}


\section{Ramanujan's sum and its analogue with respect to regular integers}

\subsection{Definitions and convolutional expressions}

Ramanujan's \cite{R} sum $c_r(n)$ is defined as
\begin{equation}
c_r(n)=\sum_{\substack{a \ppmod r\\ (a, r)=1}}\exp(2\pi ian/r). 
\end{equation}
It can be written as 
\begin{equation}\label{eq:ram-conv}
c_r(n)=\sum_{d \mid (n,r)} d\mu(r/d),  
\end{equation}
which may be considered an arithmetical expression or a convolutional expression with respect to $n$ or $r$. See \cite{A,M,Si}.

Let $\eta_k$ denote the arithmetical function defined as $\eta_k(m)=m$ if $m\mid k$, and $\eta_k(m)=0$ otherwise. 
Equation $(\ref{eq:ram-conv})$ can be written in terms of a convolution with respect to $n$ as 
\begin{equation}\label{eq:ram-conv-n}
c_r(n)  =  [\eta_r(\cdot) \mu\big(r/(\cdot)\big) \ast 1(\cdot)] (n),     
\end{equation}
where $\ast$ is the Dirichlet convolution and $1(n)=1$ for all $n\in\N$.
Equation $(\ref{eq:ram-conv})$ can also be written in terms of a convolution with respect to $r$ as 
\begin{equation}\label{eq:ram-conv-r}
c_r(n)  =  [\eta_n\ast\mu] (r).    
\end{equation}

An element $a$ in a ring $R$ is said to be regular (following von Neumann)
if there exists  $x\in R$ such that $axa=a$.
An integer $a$  is said to be regular (mod $r$)
if there exists an integer $x$ such that $a^2x\equiv a\pmod{r}$. 
A regular integer (mod $r$) is regular in the ring $\Z_r$ in the sense of von Neumann.
An integer $a$ is invertible (mod $r$) if $(a, r)=1$. 
It is clear that each invertible integer (mod $r$) is regular (mod $r$). 
See \cite{AO,HT,Tot2008}

An analogue of Ramanujan's sum with respect regular integers (mod $r$) is defined as
\begin{equation}
\overline{c}_r(n) 
= \sum_{
    \substack{a \ppmod r\\
              a \text{ regular}\ppmod r}}
    \exp(2 \pi i a n / r), 
\end{equation} 
where $n\in\Z$ and $r\in\N$. See \cite{HT}. 
The function $\overline{c}_r(n)$ also has an arithmetical (or a convolutional) expression. For this purpose we introduce some concepts. 

A divisor $d$ of $n$ is said to be a unitary divisor of $n$ (written as $d\| n$) if $(d, n/d)=1$. 
The unitary convolution of arithmetical functions $f$ and $g$ is defined as 
\begin{equation}
(f\oplus g)(n)=\sum_{d\| n} f(d) g(n/d). 
\end{equation}
For material on unitary convolution we refer to \cite{Cohen60,M,Si,V}. 

Let $r\in\N$ be fixed. 
Let $g_{r}$ denote the characteristic function of the unitary divisors of $r$, 
that is, $g_{r}(n)=1$ if $n\| r$,
and $g_{r}(n)=0$ otherwise.
Then $g_{r}(n)$ is multiplicative in $n$.
Let  $\overline{\mu}_{r}$ denote the function defined by 
\begin{equation}\label{eq:mu-r}
(\overline{\mu}_{r} \ast 1)(n) = g_{r}(n). 
\end{equation}
Then $\overline{\mu}_{r}(n)$ is multiplicative in $n$ given as follows:\\
(i)\ If $p\| r$, then $\overline{\mu}_{r}(p)=0$,
$\overline{\mu}_{r}(p^2)=-1$, $\overline{\mu}_{r}(p^j)=0 $ for $j \geq 3$.\\
(ii)\ If $p^a\| r$ with $a \geq 2$, then $\overline{\mu}_{r}(p)=-1$, 
$\overline{\mu}_{r}(p^a)=1$, $\overline{\mu}_{r}(p^{a+1})=-1$, $\overline{\mu}_{r}(p^j) = 0$    
for $j \neq 0,1,a,a+1$.\\
(iii)\ If $p\nmid r$, then 
$\overline{\mu}_{r}(p)=-1$, $\overline{\mu}_{r}(p^j)=0$ for $j\geq 2$.\\
In addition, $\overline{\mu}_{r}(1)=1$.  
 
Now, we are able to present an arithmetical expression for $\overline{c}_r(n)$. 
The function $\overline{c}_r(n)$ can be written as
\begin{equation}\label{eq:ram-ana-conv}
\overline{c}_r(n)  = \sum_{d\mid (n, r)} d \overline{\mu}_{r}(r/d).  
\end{equation}
This can also be written as the Dirichlet convolution with respect to 
the variable $n$ in the form
\begin{equation}\label{eq:ram-ana-conv-n}
\overline{c}_r(n)  =  [\eta_r(\cdot) \overline{\mu}_{r}\big(r/(\cdot)\big) \ast 1(\cdot)] (n).    
\end{equation}
A unitary convolution expression of $\overline{c}_r(n)$ with respect to the variable $r$ is presented in \cite{HT} as
\begin{equation}\label{eq:ram-ana-conv-r}
\overline{c}_{r}(n) = [c_{({\bf\cdot})}(n)\oplus 1({\bf\cdot})](r). 
\end{equation}


\subsection{Even functions (mod $r$)}\label{sec:even}

Let $r\in\N$ be fixed. 
A function $f\colon\Z \to \C$ is said to be $r$-periodic or periodic (mod $r$)  if
$f(n)=f(n+r)$ for all $n\in \Z$. 
A function $f\colon\Z \to \C$ is said to be $r$-even or even (mod $r$) if
$f(n)=f((n, r))$ for all $n\in \Z$. 
Each $r$-even function is $r$-periodic.
Ramanujan's sum $c_r(n)$  and its analogue $\overline{c}_r(n)$ are examples of $r$-even functions. 

For material on $r$-periodic and $r$-even functions we refer to 
\cite{A,C,M,Sa,SS,Si,TH}. 

\subsection{Multiplicative properties}

We here present some multiplicative properties of the usual Ramanujan's sum and its analogue with respect to regular integers (mod $r$). 
We begin with a general theorem on multiplicative $r$-even functions. 

\begin{thm}\label{th:gen-mult}
Let $f(n, r)$ be an arithmetical function of two variables. 
If for each $r\ge 1$, $f(n, r)$ is $r$-even as a function of $n$ and 
if for each $n\ge 1$, $f(n, r)$ is multiplicative in $r$, 
then $f(n, r)$ is multiplicative as a function of two variables $r$ and $n$ $(\in\N)$. 
\end{thm}

\begin{proof} See \cite{HT}. 
Let $(mr, ns)=1$. Then 
\begin{eqnarray*}
f(mn, rs) 
&=&  f(mn, r) f(mn, s) = f((mn, r), r)f((mn, s), s)\\
&=&  f((m, r), r) f((n, s), s)= f(m, r)f(n, s). 
\end{eqnarray*}
The proof is completed. 
\end{proof}

The following multiplicative properties of Ramanujan's sum are known in the literature, see e.g. \cite{AA,Hau,SH}. 

\begin{thm}\label{th:mult}
{\rm  (1)} For each $n\in\Z$, $c_r(n)$ is multiplicative in $r$.

{\rm  (2)} For each $r\in\N$, $c_r(n)$ is semimultiplicative in $n$ $(\in\N)$. 

{\rm  (3)} $c_r(n)$ is multiplicative
as a function of two variables $r$ and $n$ $(\in\N)$. 
\end{thm}

\begin{proof}
(1) Equation $(\ref{eq:ram-conv-r})$ says that $c_r(n)$ in $r$ is the Dirichlet convolution of the function $\eta_{n}(r)$ in $r$ and the M\"{o}bius function $\mu(r)$. Since these functions are multiplicative in $r$ and the Dirichlet convolution of two multiplicative functions is multiplicative \cite{M}, Ramanujan's sum $c_r(n)$ is multiplicative in $r$. 

(2) We utilize Equation $(\ref{eq:ram-conv-n})$. The functions 
$\eta_{r}(n)$ and $1(n)$ are multiplicative in $n$ and therefore they are also semimultiplicative in $n$. 
The function $\mu(n)$ is  multiplicative in $n$ and thus 
$\mu\big(r/n\big)$ is semimultiplicative in $n$. 
The usual product  and the Dirichlet convolution of semimultiplicative functions is semimultiplicative. 
This shows that $c_r(n)$ is semimultiplicative in $n$. 

(3) This follows directly from Theorem \ref{th:gen-mult}. 
\end{proof}

We next present multiplicative properties of the analogue of Ramanujan's sum with respect to regular integers (mod $r$), see \cite{HT}. 

\begin{thm}\label{th:mult-ana}
{\rm  (1)} For each $n\in\Z$, $\overline{c}_r(n)$ is multiplicative in $r$.

{\rm  (2)} For each $r\in\N$, $\overline{c}_r(n)$ is semimultiplicative in $n$ $(\in\N)$. 

{\rm  (3)} $\overline{c}_r(n)$ is multiplicative
as a function of two variables $r$ and $n$ $(\in\N)$. 
\end{thm}

\begin{proof} 
(1) Equation $(\ref{eq:ram-ana-conv-r})$ says that  
$\overline{c}_r(n)$ in $r$ is the unitary convolution of Ramanujan's sum $c_r(n)$ in $r$ and the constant function $1(r)$. Since these functions are multiplicative in $r$ and the unitary convolution of 
two multiplicative functions is multiplicative \cite{M}, we see that 
$\overline{c}_r(n)$ is multiplicative in $r$. 

(2) We utilize Equation $(\ref{eq:ram-ana-conv-n})$. 
The functions $\eta_{r}(n)$ and $1(n)$ are multiplicative in $n$ 
and therefore they are  semimultiplicative in $n$. 
The function $\overline{\mu}_{r}(n)$ is  multiplicative in $n$ and thus 
$\overline{\mu}_{r}\big(r/n\big)$ is semimultiplicative in $n$. 
The usual product  and the Dirichlet convolution of semimultiplicative functions is semimultiplicative. 
This shows that $\overline{c}_r(n)$ is semimultiplicative in $n$. 

(3) This follows directly from Theorem \ref{th:gen-mult}. 
\end{proof}

\begin{rem}
We know that Ramanujan's sum $c_r(n)$ is multiplicative in $r$ and therefore completely determined by its values at prime powers given as 
\begin{equation}\label{eq:ram-values}
c_{p^k}(n)
=\begin{cases}
p^k-p^{k-1} & \text{if $p^k\mid n$}\\
-p^{k-1} & \text{if $p^{k-1}\mid n, p^k\nmid n$}\\
0 & \text{otherwise.} 
\end{cases} 
\end{equation}
We also know that Ramanujan's sum $c_r(n)$ is   
semimultiplicative in $n$ $(\in\N)$.  
The smallest value of $n$ $(\in\N)$ for which $c_r(n)\ne 0$ is $n=r/\gamma(r)$,
where $\gamma(r)$ is the product of the distinct 
prime factors of $r$. 
This means that the constant $a$ in Definition \ref{de:semimu}
is equal to $r/\gamma(r)$ for Ramanujan's sum $c_r(n)$. 
This follows from the property $(\ref{eq:ram-values})$. 
Ramanujan's sum $c_r(n)$ is quasimultiplicative in $n$ $(\in\N)$ if and only if $a=r/\gamma(r)=1$, which means that $r$ is squarefree. 
Ramanujan's sum $c_r(n)$ is multiplicative in $n$ $(\in\N)$ if and only if
$c_r(1)=\mu(r)=1$, which means that $r$ is squarefree and
$\omega(r)$ is even, where $\omega(r)$ is the number of distinct prime 
factors of $r$. 
As a consequence of the quasimultiplicativity of $c_r(n)$ we obtain the property 
\begin{equation}\label{eq:ram-quasi}
c_r(m)c_r(n)=\mu(r)c_r(mn)\ \ {\rm if}\ (m, n)=1,   
\end{equation}
which can also be found in \cite[p. 90]{M}. 
\end{rem}

\begin{rem}\label{re:an-ram-mult}
We know that $\overline{c}_r(n)$ is multiplicative in $r$.  
Its values at prime powers $r=p^k$ are given as 
$$
\overline{c}_{p^k}(n)=1+c_{p^k}(n)
=\begin{cases}
1+p^k-p^{k-1} & \text{if $p^k\mid n$}\\
   1-p^{k-1} & \text{if $p^{k-1}\mid n, p^k\nmid n$}\\
          1 & \text{otherwise.} 
\end{cases} 
$$
We also know that $\overline{c}_r(n)$ is semimultiplicative in $n$ $(\in\N)$. 
The smallest value of $n$ $(\in\N)$ for which $\overline{c}_r(n)\ne 0$ is $n=\prod_{p\| r} p$. 
This means that the constant $a$  in Definition \ref{de:semimu} 
is equal to $n=\prod_{p\| r} p$. 
In fact, by the multiplicativity of $\overline{c}_r(n)$ in $r$ we have 
$$
\overline{c}_r(n)=\prod_{p^k\| r} (1+c_{p^k}(n)). 
$$
If $k=1$ (i.e. $p\| r$), then $1+c_{p^k}(1)=1+\mu(p)=0$ and 
$1+c_{p^k}(p)=1+\phi(p)=p\ne 0$. 
If $k\ge 2$, then $1+c_{p^k}(1)=1+\mu(p^k)=1\ne 0$. 
This shows that 
$$
n=\prod_{p\| r} p \prod_{p^k\| r\atop k\ge 2} 1
$$ 
 is 
the smallest value of $n$ $(\in\N)$ for which $\overline{c}_r(n)\ne 0$. 
The function value of $\overline{c}_r(n)$ at $n=\prod_{p\| r} p$ is also $\prod_{p\| r} p$. 
This implies that $\overline{c}_r(n)$ is multiplicative in 
$n$ if and only if $\prod_{p\| r} p=1$, which holds if and only if  $r$ is squareful  or $r=1$. Note that $\overline{c}_r(n)$ is quasimultiplicative in 
$n$ if and only if it is multiplicative in $n$. 
This shows that 
\begin{equation}\label{eq:an-ram-quasi}
\overline{c}_r(m) \overline{c}_r(n)=\overline{\mu}(r) \overline{c}_r(mn)\ \ {\rm if}\ (m, n)=1,  
\end{equation}
where $\overline{\mu}(r)$ denotes the arithmetical function such that 
$\overline{\mu}(r)=1$ if $r$ is squareful  or $r=1$, 
and $\overline{\mu}(r)=0$ otherwise. 
\end{rem}

\begin{rem} It is known \cite{Si} that if $f(n, r)$ is multiplicative as a function of two variables, then for any $r\ge 1$, $f(m, r)f(n, r)= f(1, r) f(mn, r)$ whenever $(m, n)=1$. 
Taking $f(n, r)=c_r(n)$ gives $(\ref{eq:ram-quasi})$, and taking $f(n, r)=\overline{c}_r(n)$ gives its analogue $(\ref{eq:an-ram-quasi})$. 
\end{rem}


\medskip
\noindent
{\bf Acknowledgement}\ Financial support from The Magnus Ehrnrooth Foundation is gratefully 
acknowledged.

\newpage

\end{document}